\documentclass[psamsfonts]{amsart}

\usepackage{amssymb,amsfonts}
\usepackage{mathtools}
\usepackage{color}

\newtheorem{example}{Example}

\newtheorem{lemma}{Lemma}

\newtheorem{proposition}{Proposition}
\newtheorem{remark}{Remark}

\newcommand{\dd}{\partial \overline{\partial} }
\newcommand{\p}{\partial}
\newcommand{\op}{\overline{\partial} }
\newcommand{\C}{\mathbb{C}}

\title[Local models]{Local models for conical K\"ahler-Einstein  metrics}
\author[M. de Borbon]{Martin de Borbon}
\author[C. Spotti]{Cristiano Spotti}
\address{QGM Centre for Quantum Geometry of Moduli Spaces, Aarhus University}
\email{mdb@qgm.au.dk;  c.spotti@qgm.au.dk}

\begin{document}

\maketitle

\begin{abstract}
	In this note we use the Calabi ansatz, in the context of metrics with conical singularities along a divisor, to produce regular Calabi-Yau cones and K\"ahler-Einstein metrics of negative Ricci with a cuspidal point. As an application, we describe singularities and cuspidal ends of the completions of the complex hyperbolic metrics on the moduli spaces of ordered configurations of points in the projective line introduced by Thurston and Deligne-Mostow.

\end{abstract}

\section{Introduction}

Calabi, in \cite{calabi}, considered lifts of K\"ahler metrics on a base manifold to the total space of a hermitian holomorphic vector bundle over it by imposing invariance under the unitary group of the fibre. The K\"ahler-Einstein equation for the lifted metric reduces to an ODE and Calabi studied the problem of existence of solutions which satisfy appropriate boundary conditions that guarantee a smooth extension over the zero section and completeness. A very particular case is when the base is a K\"ahler-Einstein (KE) Fano manifold and the vector bundle is a rational multiple of the canonical bundle equipped with
the induced hermitian metric. It is then easy to check that a suitable power of the norm function on the fibres is a K\"ahler potential for a Calabi-Yau (CY) cone metric on the blow-down of the total space at the zero section. These metrics are some natural candidates for tangent cones of weak KE metrics on singular varieties, see \cite{donaldsonsunII} and \cite{heinsun}.

In this note we partially extend the ansatz to the setting of K\"ahler-Einstein metrics with conical singularities along a divisor, thus providing models for tangent cones of KE metrics on klt pairs. Moreover,  we describe a lift of Ricci-flat metrics on the base to KE metrics of negative scalar curvature with a cuspidal point which seem possibly relevant in the study of asymptotics of KE metrics on log-canonical pairs as in \cite{bermanguenancia}. We also specialize our constructions to the cases of log constant holomorphic sectional curvature metrics, which we define \emph{inductively}.  As an application we describe  the asymptotic geometry of the completions of the complex hyperbolic metrics on the moduli spaces \( \mathcal{M}_{0, n+3} \) of weighted ordered configurations of points on \(\mathbb{CP}^1\) defined by Thurston \cite{thurston} and Deligne-Mostow \cite{delignemostow}. 
In the last section we return to the more general K\"ahler-Einstein setting by discussing relations to normalized volumes of valuations and proposing a Bogomolov-Miyaoka-Yau inequality in this context which characterizes the equality case with constant holomorphic sectional curvature log metrics.

\subsection*{Acknowledgments} Both authors are supported by AUFF Starting Grant 24285 and DNRF Grant DNRF95 QGM `Centre for Quantum Geometry of Moduli Spaces'. C. S. is also supported by Villum Fonden 0019098.

\section{Background}

For readers' convenience, and for fixing the notation, we briefly recall some  definitions and standard constructions for smooth K\"ahler metrics.

\subsection{\(\eta\)-Sasaki-Einstein metrics}
Let \((X, \omega)\) be a K\"ahler manifold with \( [\omega] \in 2 \pi c_1 (L^{-1}) \), where \(L^{-1}\) is an ample holomorphic line bundle over \(X\). Write \(L\) for its dual and \( \pi : L \to X \) for the projection map. Let \(h\) be a hermitian metric on \(L\) with curvature \(i \omega\), that is \( i \dd \log \rho^2 = \pi^* \omega \),  where \( \rho : L \to \mathbb{R}_{\geq 0} \) is the norm function given by \( \rho (s) = |s|_h \). On the complement of the zero section  \(L^{\times}\) we have the global angular \(1\)-form
\begin{equation}
	d \theta = \mbox{Im} (\partial \log \rho^2),
\end{equation}
which satisfies \( d (d\theta) = \pi^* \omega \). Scalar multiplication on the fibres defines a \(\C^*\)-action: there is a real holomorphic vector field \(V (p) = \frac{d}{dt} |_{t=0}  (e^{it} \cdot p) \) and \( d\theta (V) \equiv 1 \). Each level set \( \{ \rho = const >0 \} \) is a principal \(S^1\)-bundle over \(X\) and the restriction of \(i d\theta\) to it is the induced Chern connection. 

For any \( \gamma>0 \) define
\begin{equation} \label{calabiansatz1}
	\omega_C = \frac{i}{2} \dd r^2 , \hspace{3mm} \mbox{with} \hspace{3mm} r^2 = \rho^{2\gamma} .
\end{equation}
By using the identity \( \dd f = f \dd \log f + f^{-1} \partial f \wedge \overline{\partial} f \) we have that \( \omega_C = (\gamma \rho^{2\gamma} / 2) \pi^* \omega + \gamma^2 \rho^{2\gamma -1} d\rho d \theta \). The standard polar coordinates relations \( I d\rho = - \rho d\theta \) and \( I d \theta = \rho^{-1} d\rho \), where \(I\) denotes the complex structure acting on \(1\)-forms as \( I \alpha (\cdot) = \alpha (I \cdot) \),  show that the  associated metric is
\begin{equation}
	g_C = dr^2 + r^2 \overline{g} , \hspace{3mm} \mbox{with} \hspace{3mm} \overline{g} = \pi^* (\gamma g/2) + \gamma^2 d \theta^2 .
\end{equation}
We recognize \( (L^{\times}, g_C) \) as a Riemannian cone with link \( (Y, \overline{g}) \), where \( Y = \{\rho =1\}  \).

\begin{proposition} \label{P1} (e.g., \cite{boyergalicki, sparks})
\( \overline{g} \)	is a regular sasakian metric on \(Y\) with contact form \( \eta = \gamma d \theta \) and Reeb vector field \( \xi = \gamma^{-1} V \). Moreover, \( \mbox{Ric}_{g} = \mu g \) if and only if
\( \mbox{Ric}_{\overline{g}} = ( (2/\gamma)\mu -2 ) \overline{g} + (2n +2 - (2/\gamma)\mu) \eta^2,  \) i.e., \(\overline{g}\) is an \(\eta\)-\emph{Sasaki-Einstein metric}.

\end{proposition}

%\begin{proof}
	%Write \( g_{sc} = \gamma g /2 \), so that \( \overline{g} = \pi^* g_{sc} + \eta^2 \) and \( \pi : (Y, \overline{g}) \to (M, g_{sc}) \) is a Riemannian submersion whose fibers are geodesic circles of length \( 2 \pi \gamma \). The proposition follows from the following facts: \( \mbox{Ric}_{\overline{g}} (v, w) = 0 \) if \(v\) is vertical and \(w\) horizontal; for any Sasakian manifold \( \mbox{Ric}_{\overline{g}} (\xi, \xi) = 2 n \); and \( \mbox{Ric}_{\overline{g}} (\tilde{X}, \tilde{Y}) = \mbox{Ric}_{g_{sc}} (X, Y) - 2 g_{sc} (X, Y) \), where \( X, Y \) are tangent vectors to \(M\) and \( \tilde{X}, \tilde{Y} \) are their horizontal lifts. 
%\end{proof}

The proposition establishes a correspondence between KE and regular \(\eta\)-Sasaki-Einstein metrics: if \(X\) is simply connected the curvature condition \(d (d\theta) = \pi^*\omega \) determines the connection uniquely up to gauge transformations and therefore the lifted metric \(\overline{g}\) up to isometry. In the non-simply connected case we can take connections with the same curvature and different holonomy, which corresponds to taking different holomorphic line bundles \(L\) with the same first Chern class.
	
	%\item 
	%The lifts of complex hyperbolic metrics are modeled on boundaries of tubular neighborhoods of \( \mathbb{CH}^n \subset \mathbb{CH}^{n+1} \).The case \( n=1 \) corresponds to \( \tilde{SL}(2, \mathbb{R}) \) geometry. 
	 %The lifts of flat metrics are locally modeled on the Heisenberg group with its natural geometry, see \cite{goldman}. The lifts of the Fubiny-Study metric are locally isometric to tubular neighborhoods of \( \mathbb{CP}^n \subset \mathbb{CP}^{n+1} \) for a suitable value of the radius it is the round \(S^{2n+1}(1)\).

\subsection{Calabi-Yau cones} 

\begin{lemma}\label{L1} (\cite{boyergalicki, sparks}) Let \(\omega_C\) be the K\"ahler cone metric defined by Equation \ref{calabiansatz1}, then
	\begin{equation}
	\mbox{Ric}(\omega_C) = \pi^* \left( \mbox{Ric}(\omega) - \gamma(n+1) \omega \right).
	\end{equation}
	 \(\omega_C\) is Ricci-flat if and only if \( \mbox{Ric}(\omega) = \mu \omega \) for some \(\mu>0\) and \( \gamma = \mu (n+1)^{-1} \).
\end{lemma}

\begin{proof}
%We can easily compute	\( \omega_C^{n+1} = a \rho^{2\gamma(n+1)-1} (\pi^* \omega)^n d \rho d \theta \) with \( a = (n+1) \gamma^3 2^{-1} \).
We define local complex coordinates \(z_1, \ldots, z_n, w\) by taking a non-vanishing holomorphic section \( s \) and write points in \(L\) as \( p = w(p) s (\pi(p)) \). In these coordinates we have that, up to a constant factor, the volume form of the metric is %\( \rho = |w| h \) with \( h (p) = |s(\pi (p))| \), \( d \rho = |w| dh + h d|w| \), \( (\pi^* \omega)^n d\rho d\theta = (\pi^* \omega)^n h d|w| d \theta = (\pi^* \omega)^n (h/|w|) dw d \overline{w} = (\pi^* \omega)^n (\rho/|w|^2) dw d \overline{w} \). We conclude that 
\(|w|^{-2}  \rho^{2\gamma(n+1)} (\pi^* \omega)^n d w d \overline{w}. \) The claim follows from \( \mbox{Ric}(\omega) = - i \dd \log \det (g_{\alpha \overline{\beta}}) \).\end{proof}
An interesting numerical invariant for CY cones is given by the \emph{metric density}. Note that in such case \(L\) must be a rational multiple of the canonical bundle of a Fano manifold $X$ of Fano index \(I(X)\), i.e., $L=K_X^{\frac{m}{I(X)}}$ with \(m >0\), and the density is then equal to: 
\[\nu:=\frac{Vol(Y, \bar{g})}{Vol(S^{2n+1}(1))}= \frac{I(X)}{m (n+1)^{n+1}} c_1^n(X).\] 

An observation, which will be important later on,  is the following:

\begin{lemma}\label{L2}
Let \(\omega_C\) be a Calabi-Yau cone. For \( \lambda = \pm 1 \) the expression
	\begin{equation}
	\omega_{C, \lambda} = \lambda \frac{i}{2} \dd \log(1+ \lambda r^2) 
	\end{equation}
defines for  \( \lambda =1\) a KE metric  with positive  scalar curvature, resp. negative for \(\lambda = -1\) and  \(\{r< 1\}\).  The tangent cone of \(\omega_{C, \lambda}\) at the singularity is \(\omega_C\).
\end{lemma}

\begin{proof}
	We show only the case \(\lambda=1\). It is straightforward to compute:
	\begin{equation*}
		\omega_{C, 1}^{n+1} = \frac{1}{(1+r^2)^{n+1}} \omega_C^{n+1} - \frac{n+1}{(1+r^2)^{n+2}} \frac{i}{2} \p r^2 \wedge \op r^2 \wedge \omega_C^n 
	\end{equation*}  
	We use the identity \( i \p f \wedge \op f \wedge \omega_C^n = 2^{-1}(n+1)^{-1} |\nabla f|^2 \omega_C^{n+1} \) applied to \( f=r^2 \) together with \(|\nabla (r^2)|^2 = 4 r^2\) to get \( \omega_{C, 1}^{n+1} = (1+r^2)^{-n-2} \omega_C^{n+1} \). We conclude that \( \mbox{Ric}(\omega_{C, 1}) = (n+2) \omega_{C, 1} \). The statement about the tangent cone follows by expanding \( \log (1 + x) = x + (h.o.t.) \). 
	
	Note that the above argument does not use the fact that we start from  a (quasi)-regular cone. Thus it works for irregular cones too. \end{proof}

The variety \(C(Y)\) has an isolated \(\mathbb{Q}\)-Gorenstein singularity at the apex of the cone, located at \(0\). The orbits of the \(\mathbb{C}^*\)-action are totally geodesic submanifolds isometric to \(dr^2 + \gamma^2 r^2 d\theta^2\), the standard cone of total angle \(2 \pi \gamma\).  We can compactify \(C(Y)\) to a projective variety \(\overline{C(Y)} \subset \mathbb{CP}^N\) by adding a divisor at infinity \( D_{\infty} \cong X \). The metric \( \omega_{C, 1} \) extends to \(\overline{C(Y)}\) with a cone angle \(2 \pi \gamma \) along \(D_{\infty}\). Since the Fano index of \(X\) is less or equal to \(n+1\) with equality only for the projective space, \( 0 < \gamma \leq  1 \) with \( \gamma =1 \) only if \(C(Y)\) is the flat \(\C^{n+1}\).

\subsection{Negative Einstein Cuspidal Ends}
 The next lemma constructs a cuspidal negative KE metric on the disc bundle of any polarized CY manifold $(X,L, \omega)$.
\begin{lemma} \label{cusplemma}
	Assume that \( \mbox{Ric}(\omega) \equiv 0 \). Let \(t = \log (- \log \rho^2) \) and on \( \Delta^*(L) = \{0<\rho<1\} \) define \( \omega_{cusp} = - i \dd t  \), then \( \mbox{Ric}(\omega_{cusp}) = - (n+2) \omega_{cusp} \), the associated metric is complete and it has the following expression   
	\begin{equation} \label{metric cusp}
	g_{cusp} = \frac{1}{2} dt^2 + e^{-t} (\pi^* g + 2 e^{-t} d \theta^2 ) . 
	\end{equation}
	In particular, the volume of \( \{ \rho \leq \epsilon \} \) is finite as long as \( \epsilon < 1 \).  Moreover, by writing  Equation \ref{metric cusp} as  \( g_{cusp} = (1/2) dt^2 +  2 \overline{g}_t \) with \( \overline{g}_t = \pi^* (e^{-t} g /2) + e^{-2t} d \theta^2  \),  we recognize \( \overline{g}_t \)  as the \(\eta\)-Sasaki-Einstein lift of \(g\) with \( \gamma_t = e^{-t} \).
\end{lemma}

\begin{proof}
	
We have  \[ \omega_{cusp} =  (\log \rho^2)^{-2} i \p \log \rho^2 \wedge \op \log \rho^2 - (\log \rho^2)^{-1} i \dd \log \rho^2 . \]	
Since \( i \p \log \rho^2 \wedge \op \log \rho^2 = 2 \rho^{-1} d \rho d\theta \), we get	\( \omega_{cusp} = e^{-t} \pi^* \omega + 2 e^{-2t} \rho^{-1} d \rho d \theta \). Using \( dt = (\rho \log \rho )^{-1} d \rho \), Equation \ref{metric cusp} follows.  We can also compute the volume form \( \omega_{cusp}^{n+ 1} = 2 e^{-tn-2t} \rho^{-1} (\pi^* \omega)^n d\rho d \theta = 2 e^{-t(n+2)} |w|^{-2} (\pi^* \omega)^n dw d\overline{w}\), from which
\( \mbox{Ric} (\omega_{cusp}) = \pi^{*}( \mbox{Ric}(\omega)) - (n+2) \omega_{cusp}\).	
\end{proof}
If we apply this ansatz to flat space \( (\mathbb{C}^n, \omega_{euc} ) \), \(L\) is trivialized by the section identically equal to \(1\) so \( \rho^2 = e^{|z|^2} |w|^2 \). We have \( \Delta^*(L) = \{e^{|z|^2} |w|^2 <1  \} \), \( \omega_{cusp} = - i \dd \log (- \log |w|^2 - |z|^2 ) \). By considering the covering map from the upper half plane to the punctured disc \( u \to w= e^{iu} \), we get \( \log |w|^2 = -2 \mbox{Im}(u) \) and thus \( \omega_{cusp} \) is a quotient of \( \mathbb{CH}^{n+1} \).

As a final comment, the constructions performed in this section can be carried with much more generality as in \cite{hwangsinger}.

\section{Metrics with conical singularities}

Let \(X\) be a compact complex manifold, \(D \subset X\) a smooth complex hypersurface and \( 0 < \beta < 1 \). There is a well established notion of a KE  metric on \(X\) with cone angle \(2\pi\beta\) along \(D\), see \cite{donaldsoncs} and \cite{JMR}. These are smooth KE metrics on the complement of \(D\) and around \( D\) there are continuous K\"ahler potentials \( \omega_{KE} = i \dd \varphi\). More specifically, \( \varphi \in C^{2, \alpha, \beta} \) and admits a polyhomogeneous expansion. The metric \(\omega_{KE}\) is smooth in tangential directions to \(D\) and in transverse directions its singularities are modeled on the standard cone of total angle \(2\pi\beta\). The model metric is \( g_{(\beta)} = |z_1|^{2\beta-2} |dz_1|^2 + \sum_{j=2}^{n} |dz_j|^2 \) and \( C^{-1} g_{(\beta)} \leq g_{KE} \leq g_{(\beta)} \) in local coordinates where \(D = \{z_1 =0\} \). 

If \( \omega_{KE}  \in 2 \pi c_1 (L^{-1})  \) then there is a  hermitian metric on \( L \) such that \( i \dd \log \rho^2 = \pi^* \omega_{KE} \), where \( \pi : L \to X \) is the bundle projection and \( \rho (p) = |p|_h \) is the norm function as before. The norm square \(\rho^2\) is smooth on the complement of \(\pi^{-1}(D)\) and extends continuously over it. More precisely, if \( p_0 \in L^{\times} \cap \pi^{-1}(D) \) then we can define local coordinates as follows: take \( z= (z_1, \ldots, z_n) \) complex coordinates on \(U\) centered at \( \pi (p_0) \) with \( D = \{z_1 =0\} \), let \(s\) be a non-vanishing holomorphic section of \(L\)  with \(s(0) = p_0 \) and  write points \(p \in \pi^{-1}(U) \) as \( p = w (p) s(\pi (p)) \) for some \(w(p) \in \C\). Our local coordinates are \( p \to (z(\pi (p)), w(p) ) \). We have \( H (z) = |s(z)|_h^2 \) and \( \rho^2 = |w|^2 |H(z)|^2 \). The function \( \varphi = \log |H(z)|^2 \in C^{2, \alpha, \beta}\)  is a local K\"ahler potential for \(\omega_{KE}\) belonging to \(C^{2, \alpha, \beta}\), and it admits a poly-homogeneous expansion. The same regularity holds for \(\rho^2 = |w|^2 e^{\varphi} \), and similarly for \(\rho^{2\gamma}\) and \( -\log (-\log \rho^2)\). 
Alternatively, we can also appeal to the notion of \(\beta\)-smooth functions, that is smooth real functions of the variables \((z_1, \ldots, z_n, |z_1|^{2\beta})\). Polyhomogeoneous functions are \(\beta\)-smooth and the composition of a \(\beta\)-smooth function with a smooth function \(\mathbb{R} \to \mathbb{R} \) is clearly  \(\beta\)-smooth

Same as before we have the global angular \(1\)-form, \( 	d \theta = \mbox{Im} (\partial \log \rho^2) \), smooth in the complement of \( \pi^{-1}(D) \). The Chern connection in the above coordinates is \( d + \p \varphi \), since \( \varphi \in C^{2, \alpha, \beta} \) we have \( | \p \varphi | = O (|z_1|^{\beta-1}) \). If \( c_{\epsilon} \) is a loop around \(\{z_1=0\}\) that shrinks to a point as \( \epsilon \to 0 \) then \( \lim_{\epsilon \to 0} \int_{c_{\epsilon}} \p \phi =0 \), that is the connection has no holonomy along small loops that go around \({z_1=0}\). 

It is possible to give a precise definition of sasakian metrics with conical singularities in real codimension \(2\) with transverse K\"ahler structure  modelled on \(g_{(\beta)}\). If \(  \mbox{Ric}(g_{KE}) = \mu g_{KE} \) on the complement of \(D\) then for any \( \gamma > 0 \), \(\overline{g} = \pi^*(\gamma g /2) + \gamma^2 d\theta^2 \) is a conical \(\eta\)-Sasaki-Einstein metric on \( Y = \{ \rho =1 \} \) with cone angle \( 2 \pi \beta \) along \(\pi^{-1}(D)\). The Reeb vector field is \( \xi = \gamma^{-1} V \), the contact form is \( \eta = \gamma d\theta \) and \( \mbox{Ric}_{\overline{g}} = ( (2/\gamma)\mu -2 ) \overline{g} + 2(n +1 - (\mu/\gamma)) \eta^2   \) on the complement of the conical set, as in Proposition \ref{P1}. We skip the details and proceed to analyse directly the corresponding Calabi-Yau cones and cuspidal ends.

\subsection{Cones}
We discuss Calabi-Yau cone metrics with conical singularities along a divisor invariant under the \(\C^*\)-action generated by the Reeb vector field. We refer to these metrics as log Calabi-Yau cones. More precisely in this section we construct regular log Calabi-Yau cones as lifts of conical K\"ahler-Einstein metrics on log Fano pairs. Our terminology differs from the one used by Chi Li, our log Calabi-Yau cones are referred to as log Fano cones in \cite{chiliextensions}. 
 
Let \(X\) be a Fano manifold of index \(I\) together with \(L \in \mbox{Pic}(X) \) such that \(L^I = K_X\) and let \( D \in |-d L| \), for some \(d>0\). Write \(Z\) for the total space of \(L^m\) and \(C(Y) \) for \(Z\) with the zero section contracted to a point via \( p: Z \to C(Y) \), so that \(Y\) is the set of unit vectors in \(L^m\).
There is a canonical local holomorphic volume form \(\Omega_0\) on \(C(Y)\). The divisor \(D\)  determines an holomorphic function \(F\) on \(C(Y)\) with simple zeros along \( p (\pi^{-1}(D)) \). Both \(\Omega_0\) and \(F\) are unique up to constant factors. If we write \(m_{\lambda}\) for the scalar multiplication by \(\lambda\) on the fibers of \(L^m\) then  \(m_{\lambda}^{*} \Omega_0 = \lambda^{I/m} \Omega_0 \) and \(m_{\lambda}^{*}F = \lambda^{d/m} F \).

Let \( \omega_{KE} \) be a KEcs on \((X, (1-\beta)D)\). We normalize by \([\omega_{KE}] = 2 \pi c_1 (-mL) \), so there is an hermitian metric \(h\) on \(mL\) with \( \pi^{*}\omega_{KE} = i \dd \log \rho^2 \) where \(\pi\) is the bundle projection and \(\rho^2\) is the norm square function. 
Define 
\begin{equation} \label{cone def}
\omega_C = \frac{i}{2} \dd r^2, \hspace{2mm} r^2 = \rho^{2\gamma}, \hspace{2mm} \mbox{with} \hspace{2mm} \gamma = \frac{d\beta - d + I}{m} \cdot \frac{1}{n+1}.
\end{equation}

\begin{proposition} \label{Calabi ansatz prop}
	\((C(Y), \omega_C)\) is a regular log Calabi-Yau cone with 
	\begin{equation} \label{vol eq}
	\omega_C^{n+1} = a \Omega \wedge \overline{\Omega}
	\end{equation}
	where \(a\) is a constant factor and $\Omega = F^{\beta-1} \Omega_0$. Its volume density is 
	\begin{equation} \label{voldensity}
	\nu =  \frac{\mbox{Vol}(\overline{g})}{\mbox{Vol}(S^{2n+1})} = \frac{I}{m} \left( \frac{d\beta-d+I}{I} \right)^{n+1} \frac{c_1(X)^n}{(n+1)^{n+1}}.
	\end{equation}
	
	%Conversely, any regular log Calabi-Yau cone comes from a KEcs on  \((X, (1-\beta)D)\) with \(X\) Fano and \(D \in |- d L | \) as before.	
\end{proposition}

Note that the number \(\gamma\) in Equation \ref{cone def} is determined by Equation \ref{vol eq} after pulling-back both sides of it by \(m_{\lambda}\). We can easily compute the Ricci curvature of \(\omega_{KE}\) from the equation of currents
\begin{equation*}
	c_1 (X) = \mu [\omega_{KE}] + (1-\beta) [D] .
\end{equation*}
Since \( c_1 (X) = I [-L] \), \( [\omega_{KE}] = m [-L] \) and \([D] = d [-L] \) we get \( I = \mu m + (1-\beta) d \). Hence \( \mu = (d \beta -d + I) m^{-1} \) and  \( \gamma = \mu (n+1)^{-1} \) as in Lemma \ref{L1}.
\begin{proof}
	For simplicity we consider the case \(d=m=I\). We have \(D \in |-K_X| \) and let \(s \in H^0(X, [D])\)  with \(s^{-1}(0)=D\). We have  \(\pi : K_X \to X \) and  \(Z\) is the total space of \(K_X\). There is a tautological holomorphic \((n, 0)\)-form \(\phi_0\) on \(Z\) given by \(\phi_0 (\theta) = \pi^{*} \theta \). We obtain a nowhere vanishing holomorphic volume form \(\Omega_0 = \partial \phi_0 \) so that \(K_Z \cong \mathcal{O} \). In local coordinates
	\begin{equation}
	\phi_0 = \xi dz_1 \ldots dz_n , \hspace{2mm} \Omega_0 = d\xi dz_1 \ldots dz_n
	\end{equation}
	where \(\xi\) is the fiber coordinate on \(K_X\) determined by the trivializing section \(dz_1 \ldots dz_n\). 
	Our  \(C(Y) \) is \(K_X\) with its zero section contracted to a point \(o\). Since \(X\) is Fano, \(C(Y)\) is an affine variety with an isolated Gorenstein singularity and the contraction map \( p : Z \to C(Y)  \) is identified with the blow-up of \(C(Y)\) at \(o\). The zero section of \(K_X\) is identified with the exceptional divisor \( E = p^{-1}(o)\) and \( p^{-1} (p(\pi^{-1}(D))) = E + \tilde{D} \) where \(\tilde{D} = \pi^{-1}(D) \) is the strict transform of \(p(\pi^{-1}(D))\).  
	The section \(s \in H^0(X, -K_X)\) determines an holomorphic function \(F : K_X \to \mathbb{C} \) given by \(F(\theta) = \theta (s(\pi(\theta))) \). Equivalently, if \(\tilde{s}\) denotes the dual section defined by \(\tilde{s}(s) \equiv 1 \) then \(\tilde{s}\) is a meromorphic \((n, 0)\)-form with a simple pole along \(D\) and \(\theta = F(\theta) \tilde{s}(\pi(\theta)) \) for any \(\theta \in K_X\). It is clear that \((F) = E + \tilde{D} \) and therefore \( [E] + [\tilde{D}] \cong \mathcal{O} \) in \(\mbox{Pic}(Z)\).  We write \(m_{\lambda}\) for the multiplication by \(\lambda \in \mathbb{C} \) on the fibers of \(K_X\), in local coordinates \(m_{\lambda} (\xi, z_1, \ldots, z_n) = (\lambda\xi, z_1, \ldots, z_n) \). Clearly \(m_{\lambda}^{*} F = \lambda F \) and \(m_{\lambda}^{*} \Omega_0 = \lambda \Omega_0 \).
	
	\(\omega_{KE}\) is a K\"ahler-Einstein metric of positive scalar curvature with cone angle \(2\pi\beta\) along \(D\). We normalize so that \( \omega_{KE} \in 2 \pi c_1(X) \).  The KE equation is
	\begin{equation} \label{KE eq}
	i^{n^2} \tilde{s} \wedge \overline{\tilde{s}} = |\tilde{s}|_h^{2\beta} \omega_{KE}^n .
	\end{equation}
	This implies \( \mbox{Ric}(\omega_{KE}) = \beta \omega_{KE} \) on the complement of \(D\). The cone \(\omega_C\) is given by Equation \ref{cone def} with  \(\gamma = \frac{\beta}{n+1}.\)

	We now prove Equation \ref{vol eq}. We use the identity \( \dd f = f \dd \log f + f^{-1} \partial f \wedge \overline{\partial} f \) with \(f=r^2\) to obtain \(	\omega_C = r^2 \pi^{*} \tilde{\omega}_{KE} + 2 i \partial r \overline{\partial} r \),
	with \(\tilde{\omega}_{KE} = (\gamma/2) \omega_{KE} \). The pull-back of the KE equation \ref{KE eq} to \(K_X\) is \( |F|^{2\beta-2} \phi_0 \wedge \overline{\phi}_0 = \rho^{2\beta} \pi^{*} \omega_{KE}^n \) and since \(\rho^{2\beta} = r^{2(n+1)} \) we have
	\[ \omega_C^{n+1}= r^{2n} (\pi^{*}\omega_{KE}^n) \wedge \partial r \wedge \overline{\partial} r = |F|^{2\beta-2} r^{-2} \phi_0 \wedge \overline{\phi}_0 \wedge \partial r \wedge \overline{\partial} r . \]
	Equation \ref{vol eq} follows since \( \Omega_0 = \phi_0 \wedge \partial \log r \).
	
	%Let \(Y = \{r=1\} \) be the set of unit vectors (with respect to \(h\)) in \(K_X\), so we have an \(S^1\)-bundle with projection \(q: Y \to X \) and \(C(Y) \setminus \{o\} \cong \mathbb{R}_{>0} \times Y \) via \( p \to (|p|_h, p / |p|_{h} )\). Consider the \(1\)-form \(\eta = - r^{-1} J dr \) on \(C(Y)\) where \(J\) is the complex structure acting on \(1\)-forms as \( J \psi (\cdot) = \psi (J \cdot) \). Using that \( 2 \partial r = dr - i J dr \) together with equation \ref{omega cone} we get
	%\begin{equation} \label{cone calabi}
	%g_C = \omega_C ( \cdot, J \cdot ) = dr^2 + r^2\overline{g}, \hspace{2mm} \mbox{with} \hspace{2mm} \overline{g} = \eta^2 + q^{*} \tilde{g}_{KE} .
	%\end{equation}
	%The Reeb vector field is \(\xi = r J \partial_r \), so that \(\eta(\xi) \equiv 1 \). 
	
	To compute the density, note that the total volume of the link is \( \mbox{Vol} (\overline{g}) = l \cdot [\tilde{\omega}_{KE}]^n / n! \) with \( l = 2 \pi \gamma \) the length of the circle fibers and \( [\tilde{\omega}_{KE}]^n = \pi^n \gamma^n c_1(X)^n \); we recall that \([\omega_{KE}] = 2 \pi c_1(-mL) \) and \( \tilde{\omega}_{KE} = (\gamma/2) \omega_{KE} \) so that \( \mbox{Ric} (\tilde{\omega}_{KE}) = 2 (n+1) \tilde{\omega}_{KE}\). The formula follows from the fact that the volume of the round \((2n+1)\)-sphere of radius \(1\) is \( \mbox{Vol}(S^{2n+1}) = (2/n!) \pi^{n+1} \). \end{proof}  

\begin{example}
	The simplest case is \(n=1\), \(d=m=I=2\),  \( X = \mathbb{CP}^1 \), \(D = \{N\} \cup \{S\} \), \( K_X^{\times} = \mathbb{C}^2 / \{\pm\} \) and it is easy to check that \(g_C  = \mathbb{C}_{\beta} \times \mathbb{C}_{\beta} / (\pm 1) \). More generally, the cones are flat in  complex dimension  \(2\): if the metric of the base is a positively curved metric on the Riemann sphere with conical singularities at a number of points, then it lifts via the Hopf map to a polyhedral K\"ahler (PK) metric on \(\C^2\) with conical singularities at the lines corresponding to the conical points, see \cite{panov}.
\end{example} 

\begin{example}
	Let \( F \in \C [z_1, \ldots, z_{n+1}]_d \) be homogeneous  with \( D=\{F=0\} \subset \mathbb{CP}^n \) a smooth degree \(d\) hypersurface. Assume that there is a KEcs \(\omega_{KE}\) with positive scalar curvature on \( (\mathbb{CP}^n, (1-\beta)D) \). We normalize by \( \omega_{KE} \in c_1 (\mathcal{O}(1)) \). The hermitian metric on \(\mathcal{O}(-1)\) induces a continuous norm squared function \( \rho^2 : \mathbb{C}^{n+1} \to \mathbb{R}_{\geq 0} \) with \( \rho^2 (\lambda z) = |\lambda|^2 \rho^2(z) \) where \(\lambda z = (\lambda z_1, \ldots, \lambda z_{n+1})\). We let \( \omega_C = (i/2) \dd \rho^{2\gamma} \) and \( \Omega = F^{\beta-1} dz_1 \ldots dz_{n+1} \). If we let \( \gamma = (n+1)^{-1} (n+1+d\beta-d) \) then \(\omega_C^{n+1}\) is equal to \( \Omega \wedge \overline{\Omega} \), up to a  constant factor.	
	
\end{example}

Some remarks:

\begin{itemize}
	\item  The fundamental group of \(X \setminus D \) is generated by loops that go around \(D\). The connection has curvature \(d(d\theta) = \pi^{*} \omega_{KE} \) and is uniquely determined (up to gauge equivalence) by the fact that its holonomy along loops around \(D\) that shrink to a point is trivial.

	\item We can extend  Proposition \ref{Calabi ansatz prop} to a more singular setting, as simple normal crossing divisors for example. If \( (X, \sum_{j=1}^{N} (1-\beta_j) D_j) \) is a Fano klt pair endowed with a weak KEcs metric in the class \( 2 \pi c_1(-mL) \) then we set \( \gamma = (n+1)^{-1} m^{-1} (I + \sum_{j=1}^{N} (d_j \beta_j -d_j) ) \) in Equation \ref{cone def} where \( D_j \in |-d_j L| \).
	
	\item We can also compactify a Ricci-flat cone by considering \( i \dd \log (1 + r^2) \) and creating a conical angle at the divisor at infinity. If we consider \(d\) points on the Riemann sphere and \( \beta_1, \ldots, \beta_d \in (0, 1) \) satisfy the Troyanov conditions we can lift the corresponding spherical metric to a PK cone metric on \(\C^2\) with cone angle \(2\pi \beta_j\) along the respective \(d\) lines going through the origin. The compactified space is \(\mathbb{CP}^2\) and the compactified metric has positive constant holomorphic sectional curvature and cone angle \(2 \pi \gamma \) along the line at infinity, where \( 2 \gamma = 2 + \sum_{j=1}^{d} (\beta_j -1) \). We can take the cone over the line arrangement and the compactify again to get a constant positive holomorphic sectional curvature log metric on \(\mathbb{CP}^3\) and so forth. The conical hyperplane arrangements obtained in this way are log \(K\)-polystable according to \cite{fujita}.
	
\end{itemize}

\subsection{Cusps}

Let \( D \subset X \) a smooth divisor and \( \omega_{KE} \) a Ricci-flat K\"ahler metric on \(X\) with cone angle \(2\pi\beta\) along \(D\). We assume \( \omega_{KE} \in 2 \pi c_1 (-L) \) and denote by \(\rho^2\) the norm squared continuous function on the total space of \(L\) such that \( i \dd \log \rho^2 = \pi^{*}\omega_{KE} \) where \( \pi : L \to X\) is the projection map. Let \(t = \log (- \log \rho^2) \) and on \( \Delta^*(L) = \{0<\rho<1\} \) define \( \omega_{cusp} = - i \dd t  \).

\begin{proposition}\label{KEcusps}
	 \(\omega_{cusp}\) is a K\"ahler metric on \(\Delta^*(L)\) with cone angle \(2\pi\beta\) along \(\pi^{-1}(D)\) and \( \mbox{Ric}(\omega_{cusp}) = - (n+2) \omega_{cusp} \) on the complement of the conical set. The metric completion is homeomorphic to \(\Delta^*(L)\) and the volume of \( \{ \rho \leq \epsilon \} \) is finite as long as \( \epsilon < 1 \).
\end{proposition}

\begin{proof}
	On the complement of the conical set the proof of Lemma \ref{cusplemma} shows that
	\begin{equation*} 
	g_{cusp} = \frac{1}{2} dt^2 + e^{-t} (\pi^* g + 2 e^{-t} d \theta^2 )  
	\end{equation*}
	where \(d\theta\) is the global angular form determined by the hermitian metric on \(L\). Similarly, \( \mbox{Ric}(\omega_{cusp}) = \pi^* (\omega_{KE}) - (n+2) \omega_{cusp} \). The statements follow as the potential \(t\) has the desired regularity.
\end{proof}

It is also possible to extend the proposition to a more general setting, as simple normal crossing conical divisors for example, without any changes in the formulas.

\subsection{From cones to cusps}

A nice feature of the conical setting is that we can fix the ambient complex manifold and the divisor and change the conical angle to \emph{continuously change from negative to zero and positive Ricci}. 
We take \( \beta^* \) such that \( (X, (1-\beta)D) \) is log Fano if \( \beta > \beta^* \) and log Calabi-Yau if \( \beta = \beta^* \). We assume that there is a corresponding family of KEcs metrics \( g_{\beta} \) in a fixed polarization, and we denote the norm function of the hermitian metrics by \( \rho_{\beta} \) the corresponding lifted Ricci-flat cones are \( \omega_{C}(\beta) = (i/2) i \dd r_{\beta}^2 \) with \( r_{\beta}^{2} = \rho^{2  \gamma(\beta) }_{\beta} \).  The constant \( \gamma(\beta) \) depends linearly on \( \beta \) and \(  \gamma(\beta) \to 0 \) as \( \beta \to \beta^* \). To see a  cusp as \(\beta\) decreases to \(\beta^*\),  we modify the cones so that they have constant negative Ricci. Following Lemma \ref{L2},  the metrics \(  -i \dd \log (1 - r_{\beta}^2) \) are Einstein metrics of negative scalar curvature whose tangent cone at \(o\) is \(\omega_{C}  (\beta)\). The key is the elementary limit \( \lim_{\epsilon \to 0} \frac{1- \lambda^{\epsilon}}{\epsilon} = - \log \lambda \) which implies that \( \lim_{\beta \to \beta^*} \gamma(\beta)^{-1} (1- \rho^{2 \gamma(\beta)}_{\beta}) = - \log \rho^2_{\beta*} \), and therefore
\begin{equation} \label{conetocusp}
	\lim_{\beta \to \beta^*} -i \gamma(\beta)^{-1} \dd \log (1 - r_{\beta}^2) = \omega_{cusp}
\end{equation}
where \( \omega_{cusp} = - i \dd \log (- \log \rho_{\beta^*}^2 ) \).

For example, if we fix three points in the Riemann sphere then there is a continuous one parameter family of constant scalar curvature metrics \(g_{\beta}\) with cone angle \(2\pi\beta\) at the given points, the curvature is negative if \( 0 < \beta < 1/3 \),  zero if \(\beta = 1/3\) and  positive if \( 1/3 < \beta <1 \). Their sasakian lifts  form  a continuous family of \(S^1\)-invariant metrics \( \overline{g}_{\beta} \) on the \(3\)-sphere with \( \tilde{SL}(2, \mathbb{R}), \mbox{Nil} \) and spherical geometry as the cone angle \( 2 \pi \beta \) along the three Hopf circles varies. We have for \( 1/3 < \beta < 1 \) complex hyperbolic metric in \(\C^2\) with cone angle \(2\pi\beta\) along three lines through the origin, the corresponding tangent cones are PK cones with spherical link. The family converges to a metric with a cuspidal point at the origin as \( \beta \to 1/3\), corresponding to the Nil metric on the \(3\)-sphere. To continue the family beyond the cusp one must blow up the origin and consider metrics with a sharp conical angle along the exceptional divisor, as we shall see in the next section.

\subsection{Quasi-regular and irregular cases}
The Calabi ansatz can be extended to lift KEcs metrics on orbifolds, same as the smooth case. The prototypical example is the case of the cusp \( \{ w^2 = z^3 \} \subset \C^2 \) in the range \( 1/6 < \beta < 5/6 \), we obtain a PK cone metric as a lift of a spherical metric with only one conical singularity on the orbifold \( \mathbb{CP}(2, 3) \). Same applies to the more general case of \(\eta\)-Sasaki-Einstein metrics, the sasakian manifold has the structure of a Seifert fibered space. On \( \mathbb{CP}^1 \) we have a \(1\)-parameter family of  constant scalar curvature metrics \(g_{\beta}\), \( 0 < \beta < 5/6 \), with fixed volume and three conical points of angle \(\pi\), \(2\pi/3\) and \(2\pi\beta\). The scalar curvature is negative if \( 0 < \beta < 1/6 \), zero if \( \beta = 1/6 \) and positive if \( 1/6 < \beta < 5/6 \). We can lift the family with a Seifert map to  metrics \( \overline{g}_{\beta} \) on the \(3\)-sphere with cone angle \(2\pi\beta\) along the trefoil knot with \(\tilde{SL}(2, \mathbb{R})\), Nil and spherical geometry according to the angle. 

\begin{example}
	
	It is shown by Li-Sun \cite{lisun} that if \(\beta \in (1/4, 1) \) then \(\mathbb{CP}^2\) has a KEcs with cone angle \(2\pi\beta\) along a smooth conic. By Proposition \ref{Calabi ansatz prop} we can lift it to a (regular) log CY cone metric \(g_{CY}\) on \(\mathbb{C}^3\) with angle \( 2\pi \beta\) along \(\{z_1^2 + z_2^2 + z_3^2 =0 \}\). On the other hand we have a branched cover from the \(A_2\)-singularity \(\{z_1^2 + z_2^2 + z_3^2 + z_4^3 =0 \} \subset \C^4 \) to \(\C^3\) given by \( (z_1, \ldots, z_4) \to (z_1, z_2, z_3) \). The pull-back of \(g_{CY}\) is a (quasi-regular) log CY  cone metric on the \(A_2\)-singularity with cone angle \(2 \pi \tilde{\beta} \) with \(\tilde{\beta} = 3 \beta\) along \(\{z_4=0\}\). In particular, the pull-back is a smooth CY cone metric when \(\beta=1/3\).
	
\end{example}

We expect to be plenty of irregular examples, e.g., in the toric setting as in \cite{futakionowang}. The case of cuspidal ends corresponds to \(\eta\)-Sasaki-Einstein metrics whose transverse geometry is log Calabi-Yau; there does not seem to be any obstruction to solve the Monge-Amp\`ere equation in this setting, since the transverse geometry is Calabi-Yau as in  \cite{elkacimi}.

\section{Constant holomorphic sectional curvature log metrics}

Some precedents of what we are going to discuss are the following:

\begin{itemize}
	\item Polyhedral K\"ahler manifolds, introduced by Panov \cite{panov}. These are polyhedral spaces with unitary holonomy. One difficulty in adapting this notion to the non-zero curvature case is the lack of totally geodesic real hypersurfaces in \(\mathbb{CP}^n\) and \( \mathbb{CH}^n \); these are replaced by extors, see \cite{goldman}. It is also not clear that the complex structure defined on the smooth locus extends, so far this has only been proved for \(\dim_{\mathbb{C}} =2 \).
	
	\item  \((X, G)\)-cone manifolds, introduced by Thurston \cite{thurston}. This is a much broader concept. \(X\) is a homogeneous space and \(G\) is a subgroup of the isometry group, an \((X, G)\)-manfold is given by charts to \(X\) with transition functions in \(G\). The notion of an \((X, G)\)-cone manifold is worked inductively. The result is a stratified space, locally isometric to \(X\) in the smooth locus and with totally geodesic strata. Even if \(G\) is contained in a unitary group, it is not always the case that there is a global complex structure. 
\end{itemize}

We take a \emph{complex view-point} and we always restrict to the case of angles less than \(2\pi\). We start with a log-canonical pair \( (X, \sum_{j=1}^{N} (1- \beta_j)D_j ) \) and we want to define the notion of a constant holomorphic sectional curvature K\"ahler metric with cone angle \(2\pi\beta_j\) along \(D_j\). In the complement of the divisors we want a smooth metric locally isometric in complex coordinates to the respective model \( \mathbb{CH}^n, \mathbb{C}^n, \mathbb{CP}^n \). The point is to give models around the singular locus, this imposes strong restrictions on the pair. 

The notion when \( \dim_\C =1 \) is clear, we have the models \( -i \dd \log (1 - |z|^{2\beta}) \), \(i \dd |z|^{2\beta}\), \(i \dd \log (1 + |z|^{2\beta}) \) for the conical points and \( - i \dd \log (- \log |z|^2) \) for the cusps. We proceed inductively, constructing local model metrics as lifts of positive and zero constant holomorphic curvature conical (orbifold) metrics via the Calabi ansatz. 

We assume first that \( 0 < \beta_j < 1 \) and that the pair is klt, to avoid the cuspidal behaviour. It is enough to describe the tangent cones, which must be  polyhedral K\"ahler (PK) cones. For example, at the simple normal crossing locus of the support of \(D\) written in local coordinates as \(\{z_1 \ldots z_k =0\}\) the tangent cone must be \(\C_{\beta_1} \times \ldots \times \C_{\beta_k} \times \C^{n-k} \) and the corresponding non-zero constant holomorphic sectional curvature models are:  
\begin{equation*}
	- i \dd \log (1 - (\sum_{j=1}^{k}|z_j|^{2\beta_j} + \sum_{j=k+1}^{n}|z_j|^2) ),  \hspace{2mm} i \dd \log (1 + \sum_{j=1}^{k}|z_j|^{2\beta_j} + \sum_{j=k+1}^{n}|z_j|^2) .
\end{equation*}
Any irregular cone splits as \(\C_{\beta} \times C(Y) \) with  \(C(Y)\) a lower dimensional PK cone, see \cite{panov}. Otherwise the cone is quasi-regular and is the lift of a constant positive holomorphic curvature metric.  If \( i \dd r^2 \) is a PK cone, then \( - i \dd \log (1-r^2) \) and \(i \dd \log (1+r^2) \) are the constant holomorphic sectional curvature models which have the given PK cone as a tangent cone. We conclude that in the klt setting is enough to specify the tangent PK cone at any point, and these can be described inductively. Finally, one can include cuspidal points in the same way. To our knowledge, there do not seem to be compact examples of constant curvature and non-isolated cuspidal singularities.

\subsection{Complex hyperbolic metrics on \(\mathcal{M}_{0, n+3}\)}
We recall the following construction from Section 4 in \cite{delignemostow}.  Let \( \mu_1, \ldots, \mu_{n +3} \in (0, 1) \) be such that \( \sum_{j=1}^{n+3} \mu_j =2 \). Let \(M^{st} \subset (\mathbb{CP}^1)^{n+3} \) be the set of \( (x_1, \ldots, x_{n+3}) \) such that \( \sum_{i: x_i =z} \mu_i < 1 \) for all \( z \in \mathbb{CP}^1 \). Note that \( M^{st} \) contains the set \( M_0 \) where \( x_i \neq x_j \) for \( i \neq j \). The group of M\"obius transformations \( PSL(2, \mathbb{C}) \) acts  diagonally on \(M^{st}\) and we define \( X^{st} = M /PSL (2, \mathbb{C}) \). Define \(X^{ps}\) as the set of partitions into two disjoint sets \(B_1 \sqcup B_2 = \{1, \ldots, n+3\} \) with \( \sum_{j \in S_1} \mu_j = \sum_{j \in S_2} \mu_j =1 \).  Let \( X = X^{st} \sqcup X^{ps} \),  \(X^{st}\) is the set of the stable points and \(X^{ps}\) is the strictly polystable locus. The space \(X\) is a compact complex variety of complex dimension \(n\), it is a smooth complex manifold on \(X^{st}\) and is locally biholomorphic to the cone over the Segre embedding of \( \mathbb{CP}^{|B_1| -2} \times \mathbb{CP}^{|B_2|-2} \) into \( \mathbb{CP}^{(|B_1|-1)(|B_2-1)-1} \) at \(X^{ps}\). In particular, it is smooth if either \( |B_1| =2 \) or \( |B_2| =2 \). Each \( \{i, j\} \subset \{1, \ldots, n+3\}  \)  such that \( \mu_i + \mu_j < 1 \) defines a complex hypersurface \( D_{ij} \subset X \) equal to the projection of \( \{ x_i = x_j \} \). Note that \( X \setminus \cup_{i, j} D_{ij} \) is biholomorphic to the space of configurations of  \emph{ordered} \((n+3)\)-tuples of distinct points in the projective line \( \mathcal{M}_{0, n+3} := M_0 / PSL (2, \mathbb{C}) \). We denote by \(D_{ij}^{\times}\) the set of `generic' points on the divisor, that is \( D_{ij} \setminus \cup_{\{k, l\}\neq \{i, j\} } D_{k l} \).  We have \( D_{ij}^{\times} \cong \mathcal{M}_{0, n+2} \) and \(D_{ij} \cap X^{st} \) is smooth as it corresponds to the set of stable points when the weights \(\mu_i\) and \(\mu_j\) are replaced by the single \( \mu_i + \mu_j \) weight. The pair \( (X, \sum_{i < j} (\mu_i + \mu_j) D_{ij} ) \) is log-canonical and is klt precisely on the complement of the finite set \( X^{ps} \), moreover \( K_X + \sum_{i<j} (\mu_i + \mu_j) [D_{ij}] > 0 \). 
When all \(\mu_j\) are rational  \(X\) agrees with a weighted GIT quotient.

Each \((x_1, \ldots, x_{n+3}) \in \mathcal{M}_{0, n+3} \) defines a unique up to scale a compatible flat metric on \(\mathbb{CP}^1\) with cone angle \(2\pi(1-\mu_j)\) at \(x_j\). The space \(\mathcal{M}_{0, n+3}\) is identified with the space of polyhedral metrics of fixed area on the \(2\)-sphere with \(n+3\) marked vertices of prescribed conical angles up to orientation preserving isometries that respect the markings. The polyhedral metrics induce a natural Weil-Petersson metric on \(\mathcal{M}_{0, n+3}\) which turns out to be complex hyperbolic (compare Section $2$ in \cite{delignemostow}). More precisely, in local complex coordinates we map \( z= (z_1, \ldots, z_{n}) \to (z_1, \ldots, z_{n}, 0, 1, \infty) \) and let
\begin{equation}
	\Omega_z = t^{-\mu_{n+1}} (t-1)^{-\mu_{n+2}} \prod_{j=1}^{n} (t-z_j)^{-\mu_j} dt ,
\end{equation}
which extends as a locally defined meromorphic one-form on \(\mathbb{CP}^1\) with a pole of order \(  \mu_{n+3} \) at \(\infty\). The Weil-Petersson metric is \( \omega_{WP} =- i \dd \log \int_{\mathbb{CP}^1} i \Omega_z \wedge \overline{\Omega_z} \). The local K\"ahler potentials extend continuously to \( X^{st} \) and \(\omega_{WP}\) extends with conical singularities of cone angle \( 2 \pi (1- \mu_i - \mu_j) \) along the generic boundary points  \( D_{ij}^{\times} \), see Proposition $3.5$ in  \cite{thurston}. 

The collision of \( k + 1 \) points defines a strata of complex codimension \(k\). We first want to examine the case of maximal codimension equal to \(n\), this corresponds to a point in \( p \in X\) representing the collision of \(n+1\) points. Consider first the case that \( p \in X^{st} \) and without loss of generality we assume that \( \sum_{j=1}^{n+1} \mu_j < 1 \). We must then have \( \mu_{n+2} + \mu_{n+3} > 1 \) and the point \(p\) represents the double of an euclidean triangle with interior angles \( \pi (1 - \sum_{j=1}^{n+1}\mu_j )\),  \(\pi (1- \mu_{n+2})\), \( \pi (1-\mu_{n+3} ) \). We have \( \binom{n+1}{2} \) smooth divisors that intersect pairwise transversely at \(p\), we can choose coordinates to \( \mathbb{C}^{n} \) centered at \(p\) in a way that each divisor is represented by a complex hyperplane going through the origin. This defines a (rigid) hyperplane arrangement \( \mathcal{H} = \cup_{i, j} H_{ij} \) in \(\mathbb{CP}^{n-1}\): given \(n+1\) points in general position in \(\mathbb{CP}^{n-1}\), \(H_{ij}\) is the hyperplane that goes through all the points but not through the \(i\)-th and \(j\)-th. When \( n=2\) this is just three points in the Riemann sphere, when \( n=3 \) the arrangement is known as the complete quadrilateral. The pair \( ( \mathbb{CP}^{n-1}, \sum_{i < j} (\mu_i + \mu_j) H_{ij} ) \) is log-Fano and by comparing with Section $3$ of \cite{thurston}, it gives a metric with positive constant holomorphic sectional curvature with cone angle \( 2 \pi (1 - \mu_i - \mu_j) \) along \(H_{ij}\). The tangent cone of \( \omega_{WP} \) at \(p\) is the polyhedral K\"ahler cone with the above metric as a complex link.

More generally the tangent cones at points in \(X^{st}\)  decompose as products and can be described inductively. For example, if \( \{i, j\} \cap \{k, l\} = \emptyset \) and \( \mu_i + \mu_j < 1, \mu_k + \mu_l < 1 \) then \( D_{ij} \cap D_{kl} \) is a complex codimension two variety and the tangent cone at the generic points \( (D_{ij} \cap D_{kl})^{\times} \) is \( \C_{1- \mu_i - \mu_j} \times \C_{1-\mu_k - \mu_l} \times \C^{n-2} \).   More systematically, following \cite{mcmullen} we let \(\mathcal{P}= \sqcup_{\alpha}B_{\alpha} \) be a partition of \( \{1, \ldots, n+3\} \) such that \( \sum_{j \in B_{\alpha}} \mu_j < 1 \) for each \(\alpha\). Such a partition defines  \( X^{\mathcal{P}} \subset X\) given by the collision of the points labeled by the sets \(B_{\alpha}\). Note that \(|\mathcal{P}| \geq 3\) and \(X^{\mathcal{P}}\) is a point if \( |\mathcal{P}|=3 \), representing the double of an euclidean triangle. The strata are totally geodesic submanifolds and the induced metric agrees after identifying \( X^{\mathcal{P}} \cong \mathcal{M}_{0, |\mathcal{P}|} \) with the complex hyperbolic metric given by the weights \( \sum_{j \in B_{\alpha}} \mu_j \). For each \(\alpha\) we have a constant positive holomorphic sectional curvature metric on \(\mathbb{CP}^{|B_\alpha|-2}\) that lifts to a PK cone metric \(g_{\alpha}\) on \(\C^{|B_{\alpha}|-1}\). %with volume density \((1- \sum_{j \in B_{\alpha}}\mu_j )^{|B_{\alpha}|-1}\).

We summarize the above discussion in the following:

\begin{proposition} \label{posandpk} 
	Let \( \{v_1, \ldots, v_{d+2}\} \subset \mathbb{CP}^d \) be in general position and let \(\{\mu_k\}_{k=1}^{d+2}\) with $0 < \mu_k< 1$. Define
	\[H_{ij} = \mbox{span} \{v_1, \ldots, \hat{v_i}, \hat{v_j}, \ldots, v_{d+2} \}. \]
	We have:
	\begin{enumerate}
		
		\item If \( \sum_{k=1}^{d+2} \mu_k < 1 \) then there is a metric with constant positive holomorphic sectional curvature on \(\mathbb{CP}^d\) and cone angle \(2\pi(1-\mu_i - \mu_j) \) along \(H_{ij}\). It lifts via the Hopf map to a PK cone metric on \(\C^{d+1}\) with volume density equal to \( (1- \sum_{k=1}^{d+2} \mu_k)^{d+1} \), as in Proposition \ref{Calabi ansatz prop}.
		
		\item   Let \(\mathcal{P}= \sqcup_{\alpha}B_{\alpha} \) be a partition of \( \{1, \ldots, n+3\} \) such that \( \sum_{k \in B_{\alpha}} \mu_k < 1 \) for each \(\alpha\). If  \( p \in X^{\mathcal{P}} \), then the tangent cone of the complex hyperbolic metric on \(X\) at \(p\) is
		\[ \left( \prod_{\alpha} g_{\alpha} \right) \times \mathbb{C}^{|\mathcal{P}|-3}, \]
		where \(g_{\alpha}\) are the PK cones corresponding to the sets \(B_\alpha\) as in point \((1)\). Its volume density is  \( \prod_{\alpha} (1- \sum_{k \in B_{\alpha}}\mu_k )^{|B_{\alpha}|-1}\).

	\end{enumerate}
	
\end{proposition}

Note that, by using the first Chern class of the hyperplane section to identify \(H^2_{dR}(\mathbb{CP}^d) \cong \mathbb{R} \), we have that \( c_1 (\mathbb{CP}^d, \sum_{i < j} \mu_{ij} H_{ij}) = (d+1) (1 - \sum_{k=1}^{d+2}\mu_k) \), where \( \mu_{ij} = \mu_i + \mu_j \) and \( \sum_{i < j} \mu_i + \mu_j = (d+1) \sum_{k=1}^{d+2}\mu_k \), and the pair is indeed log Fano. The formula for the volume density follows immediately from Equation \ref{voldensity} and it agrees with the one given by Thurston in Proposition 3.6 in \cite{thurston}. If we consider moduli of unmarked points then we must quotient by permutations of points of equal angle. The complex links might be orbifold quotients of the above metrics, lifted to the real link of the cone via a Seifert fibration.

Now we discuss the case  \( p \in X^{ps} \). We consider the case   \( p \in X^{ps} \) is a smooth point in the moduli space. Without loss of generality we can assume that \( B_2=\{n+2,n+3\}\) and  \( \sum_{k=1}^{n+1} \mu_k = 1 \). We can slightly change the cone angles, so that  \( \sum_{k=1}^{n+1} \mu_k < 1 \). By continuity of the Weil-Petersson metric under  change of parameters and Equation \ref{conetocusp},  we conclude  that \( \omega_{WP} \) is a cuspidal metric  $- i \dd \log (- \log (t^2)) $  around \(p\) with complex link the log Calabi-Yau pair \( ( \mathbb{CP}^{n-1}, \sum_{i < j} (\mu_i + \mu_j) H_{ij} ) \) endowed with a polyhedral K\"ahler metric with cone angle \( 2 \pi (1 - \mu_i - \mu_j) \) along \(H_{ij}\). Thus:

\begin{proposition}\label{P5} 	Let \( \{v_1, \ldots, v_{d+2}\} \subset \mathbb{CP}^d \) be in general position and let \(\{\mu_k\}_{k=1}^{d+2}\) with $0 < \mu_k< 1$. Set 
	\(H_{ij} = \mbox{span} \{v_1, \ldots, \hat{v_i}, \hat{v_j}, \ldots, v_{d+2} \}. \) 
	\begin{itemize}
	
	\item If \( \sum_{k=1}^{d+2} \mu_k = 1 \) then there is a PK metric on \(\mathbb{CP}^d\) with cone angle \(2\pi(1-\mu_i - \mu_j) \) along \(H_{ij}\). It lifts to a cuspidal complex hyperbolic metric on the punctured ball \( B \setminus \{0\} \subset \C^{d+1}\) with finite volume around zero, as in Proposition \ref{KEcusps}.
		
		\item If \( \mathcal{P} = B_1 \sqcup B_2 \) with \(|B_2|=2\), then the complex hyperbolic metric on \(X\) at the point represented by \(\mathcal{P}\) is a cusp with complex link the  PK metric above on \(\mathbb{CP}^{|B_1|-2}\).
		
	\end{itemize}
\end{proposition}

More generally, if  \( \mathcal{P} = B_1 \sqcup B_2 \) with \( \sum_{k \in B_1} \mu_k = \sum_{k \in B_2} \mu_k =1 \), then the complex hyperbolic metric on \(X\) at the point represented by \(\mathcal{P}\) should be a cusp with complex link the product of the  PK metrics on \(\mathbb{CP}^{|B_1|-2}\) and \( \mathbb{CP}^{|B_2|-2} \) defined by \(B_1\) and \(B_2\) respectively.  For example if we consider the space of six conical points with all \( \mu = 1/3 \) then the partition \( \{1, 2, 3\}, \{4, 5, 6\} \), say, defines a cuspidal point in \(X\) with a neighbourhood biholomorphic to the three dimensional ordinary double point. We expect that around the cuspidal point \(\omega_{WP}\) is a cusp whose complex link is a product of two flat \( \mathbb{CP}^1 \), each of the factors being the double of an equilateral triangle.

\begin{remark}
Complex conjugation induces an isometric anti-holomorphic involution on \(X\), the set of fixed points is a totally geodesic totally real hyperbolic cone manifold of real dimension \(n\). It corresponds to points parametrizing doubles of polygons in the euclidean plane or equivalently points in the Riemann sphere lying on a circle. It is often easier to use a family of polygons to realize certain degenerations: for example the case of six conical points of equal angle converging to the cuspidal point can be realized as a sequence of hexagons degenerating to a trapezoid, and finally stretching the trapezoids by bringing together the two parallel sides. 
\end{remark}

The metrics of Proposition \ref{posandpk} are connected, by varying the cone angles, to the complex hyperbolic metrics on configuration spaces. More precisely we have \( \mathcal{M}_{0, d+3} \cong \mbox{Bl}_{v_1, \ldots, v_{d+2}} \mathbb{CP}^d \setminus (\cup_{i, j} \tilde{H}_{ij} \cup_{k=1}^{d+2} E_k)  \) where \( \tilde{H}_{ij} \) is the proper transform of \(H_{ij}\) and the \(E_k\) are the exceptional divisors. If we consider the space of \(d+3\) points in \(\mathbb{CP}^1\) with weights such that \(\mu_{d+3} + \mu_j > 1 \) and \(\mu_i + \mu_j < 1\) for any \( 1 \leq i, j \leq d+2 \), then \(X= \mathbb{CP}^d\) (at least when \(\sum_{k=1}^{d+2}\mu_k\) is sufficiently close to \(1\)) and \(D_{ij}=H_{ij}\) is the hyperplane arrangement \(\mathcal{H}\). We obtain a \(d+2\) dimensional family of complex hyperbolic metrics parametrized by \(\mu_1, \ldots, \mu_{d+2}\) with \( 1 < \sum_{k=1}^{d+2} \mu_k < 2 \), since \( 0 < \mu_{d+3} = 2 - \sum_{k=1}^{d+2} \mu_k < 1\), which extends continuously the family of metrics in Proposition \ref{posandpk} when \( 0 < \sum_{k=1}^{d+2} \mu_k \leq 1 \). If \( 1 < \sum_{k=1}^{d+2} \mu_k < 1 + \epsilon \) then the pair \( ( \mathbb{CP}^d, \sum_{i < j} \mu_{ij} H_{ij} ) \) is klt, as the weights change and the sum increases we might create points where the pair is strictly log canonical and the metrics develop a cusp, the family extends by changing the space \( X = \mathbb{CP}^d  \) blowing up the point and introducing a small angle at the exceptional divisor.

For a concrete example, when \(d=2\) we have constant positive holomorphic sectional curvature metrics with cone angle \(2\pi\beta\) along the complete quadrilateral for \( 1/2 < \beta < 1 \). At \(\beta=1/2\) we have a PK metric and for \(1/3 < \beta < 1/2\) we have complex hyperbolic metrics. At \( \beta = 1/3 \) we have a complex hyperbolic metric with four cuspidal points at the triple points of the arrangement, when \( \beta< 1/3 \) the space changes to the blow up of \(\mathbb{CP}^2\) at the triple points (that is, the moduli becomes the del Pezzo surface of degree four) and the complex hyperbolic metrics have cone angle \(2\pi \beta\) along the proper transform of the six lines and cone angle $2\pi \gamma$, with \(2\gamma = 1-3\beta\), along the four exceptional divisors. Restricting the metrics in the family to a small sphere centered at a triple point exhibits a family of metrics on the \(3\)-sphere with spherical, Nil and \(\tilde{SL}(2, \mathbb{R})\) geometry as the conical angle \(2\pi\beta\) along the three Hopf circles varies. Moreover, note that such examples describe how the geometry changes continuously in the Weil-Petersson metrics, under a birational transform of moduli spaces induced by variation of GIT quotients.

\subsection{Hypergeometric functions}

The periods of \(\Omega_z\) along the edges of two disjoint trees with vertices at \( \{ z_1, \ldots, z_{n}, 0, 1, \infty \} \)
define a map \(F\) to the cone of positive vectors in \( \C^{1, n} \), the metric is the pull-back of the complex hyperbolic metric on the ball in \( \mathbb{P} (\C^{1, n}) \); that is \( \omega_{WP} = i \dd \log ( \langle F, F \rangle_{1, n} ) \). The classical case is the one of four points in the projective line. The periods \( \int_0^1 \Omega_z, \int_{z}^{\infty} \Omega_z \) are two linearly independent solutions of the hypergeometric equation. The quotient \( \int_0^1 \Omega_z / \int_{z}^{\infty} \Omega_z \) maps the upper half plane to a circular triangle. Our restriction on the \(\mu\) parameters implies that the triangle is hyperbolic and its double is the complex hyperbolic metric described above. For more than four points, the periods of \(\Omega_z\) are solutions of higher dimensional extensions of the hypergeometric equation and the K\"ahler potentials for the complex hyperbolic metrics are explicit in terms of quotients of corresponding  hypergeometric functions.  The same discussion applies not only for the negatively curved metrics on the moduli space but also for the flat and positively curved metrics on the complex links. 

Given a unitary reflection group \( G \subset U (n) \) one can construct and explicit biholomorphism \( F: \mathbb{C}^n / G \to \C^n \) with invariant polynomials as components. We can push-forward the euclidean metric to obtain a polyhedral K\"ahler cone with conical angles along the image under \(F\) of the reflecting hyperplanes of \(G\). If we write \( z = F(u) \), the K\"ahler potential for the PK cone is \( |u|^2 \) and it can be written as a explicit function of \(z\). As a consequence one can write down explicit expressions for quotients of solutions of the corresponding differential equations in this orbifold regime. 
The same apply to reflection groups \( G \subset PU (n, 1) \). For example the double of an equilateral triangle give us a metric \(g_{flat}\) on \( \mathbb{CP}^1 \). We can lift it to a complex hyperbolic metric \(g_{cusp}\) in the unit ball in \( \mathbb{C}^2 \) with a cuspidal point at the origin and cone angle \(2\pi/3\) along three lines that go through the origin. There is a reflection group \( G \subset PU(2, 1) \) and an explicit identification \(B \cong B/G\) such that \( g_{cusp} \) is the quotient metric, see \cite{yoshida}. As a consequence one obtains an explicit expression for the Riemann mapping from the triangle to the upper half plane, whose inverse is the developing map for \(g_{flat}\)  given by the Schartz-Christoffel integral \( \int_0^z t^{-2/3} (1-t)^{-2/3} dt \).

\section{K\"ahler-Einstein equations}
\subsection{Volume densities and tangent cones}

The CY cones constructed from log KE metrics via Calabi ansatz are expected to appear more generally as models for tangent cones of  klt varieties with weak KE metrics,  when the tangent cone is smooth and (quasi)-regular. More generally, one should consider irregular situations and jumping to singular cones.  Moreover, the existence of such log CY metrics should be also equivalent to a notion of log K-polystability in the sasakian case similarly to \cite{collinszek}. However, we should  remark that, even if expected, the generalization of the results of Donaldson-Sun for tangent cones of limit spaces to the logarithmic setting are still missing. On the other hand, the work of Li \cite{chilivaluations} on characterizing the tangent cones fully algebro-geometrically via the infimum of the normalized volume of valuations extends immediately to the logarithmic setting and, in certain cases, such invariant can be checked to be equal to the volume density \cite{chiliextensions}.

\subsection{Stratified K\"ahler metrics}
 
Here we want to give a possible notion of a KEcs metric \(g_{KE}\) on a klt pair \((X, D)\) which incorporates the existence of metric tangent cones, similarly to the one given in the case of constant holomorphic sectional curvature log metrics. The tangent cones are singular, potentially irregular, CY cones and its metric singularities can be described inductively. Given a point \(p \in X \) the theory of normalized volumes of valuations already tells us the right tangent cone \((C(Y), g_C)\) one should consider. We want to express this in differential-geometric terms, we divide the task into three steps of increasing degree of generality. \emph{Step 1:} There is a biholomorphism \(F\) from  a neighbourhood of the apex of \(C(Y)\) to a neighbourhood of \(p\), we also assume that \(F\) matches the log sets. In this case we can require that \( |F^{*}g_{KE} - g_C |_{g_C} = O(r^{-\mu}) \) for some \(\mu>0\). A curious fact is that the global geometry of the pair influences and one needs to modify \(F\) by composing with a suitable automorphism of \(C(Y)\). \emph{Step 2:} We have to replace \(F\) by a diffeomorphism `approximately holomorphic' in some asymptotic sense. This is the case for example when \( \dim_{\mathbb{C}} =2 \) and the support of \(D\) has an ordinary \(d\)-tuple plane complex curve singularity which is not linearisable via an holomorphic change of coordinates to a collection of \(d\) lines going through the origin. \emph{Step 3:} Finally one must consider the case that \( (C(Y), o)\) is not locally homeomorphic to \((X, p)\), as in \cite{szekelyhidi}; the coordinate \(F\) should be replaced by a one parameter family of maps that approximate the geometry. This is the case for example of the cuspidal point \(p=0\) of the pair \( (\C^2, (1-\beta)\{w^2=z^3\}) \)  with \( 5/6 < \beta < 1 \). We have \( C(Y) = \C \times \C_{\gamma} \). There is \( u: \C \to \mathbb{R} \) with \( u(\xi) = u (-\xi) \) such that \(i\dd u>0\), \( i \dd u = |\xi-1|^{2\beta-2} |\xi+1|^{2\beta-2} d\xi d\overline{\xi} \)  in some ball and \( u = |\xi|^{2\gamma} \) outside a compact set. Define \( \omega = i \dd \varphi \) with \( \varphi = u (wz^{-3/2}) \). Then \(\omega\) is a K\"ahler metric in a neighbourhood of \( 0 \in \C^2 \) and, when \( 5/6 < \beta < 1 \), its tangent at \(0\) is \(C(Y)\). One would expect to perturb \(\omega\) to a Ricci-flat metric preserving the asymptotics. 

Having sorted a reasonable definition one then would like to understand the linearization of the KE equation. More precisely we want suitable spaces for which \( i \dd \circ \Delta^{-1} \) is a bounded operator, where \(\Delta\) denotes the Laplacian of the metric. There are some general results on analysis on stratified Riemannian manifolds but the K\"ahler case has many special features, see for instance \cite{donaldsoncs} and \cite{heinsun}. Of course, such existence theory for KE \emph{with these prescribed metric singularities} seems out of reach at the moment. Yau's second order estimate breaks in a fundamental way, as the curvature of a non-flat Calabi-Yau cone blows-up from both sides at the origin with quadratic growth. From a different point of view, such type of results can be interpreted as regularity results for weak KE metrics. At the moment, one special doable case is to restrict to two dimensional flat tangent cones \cite{deborbonspotti}; another is to assume a smoothability condition and appeal to Riemannian convergence theory and recent results for Gromov-Hausdorff limits of KE metrics \cite{heinsun}.

\subsection{Chern-Weil formulas}
Let \(\omega\) be a K\"ahler metric and let \( c_k (\omega) \) be the standard representative of the \(k\)-th Chern class in terms of the curvature of \(\omega\). In general terms we expect that the integrals \(\int_X c_k(\omega) \wedge \omega^{n-k} \) are affected by the singularities of \(\omega\) only up to co-dimension \(k\), as suggested by the following: 

\begin{proposition}
	Let \( \omega_C = (i/2) \dd r^2 \) be a non-flat cone of complex dimension \(n\) singular only at the apex. Then \(c_k (\omega_C) \wedge \omega_C^{n-k} \) is locally integrable in a neighborhood of the apex of the cone if and only if \( k < n \). 
\end{proposition}

\begin{proof}
	Write \(F\) for the Riemann curvature tensor of \(\omega_C\) regarded as a matrix valued two-form. We have \( |F|_{\omega_C} = O(r^{-2}) \) as \( r \to 0 \) and therefore, for the \(k\)-th product,   \( |\mbox{Tr}(F \wedge \ldots \wedge F)|_{\omega_C} = O (r^{-2k}) \). We have \( \omega_C = d \eta \) with \( \eta = O(r) \) and, if \( k  < n \), \( c_k (\omega_C) \wedge \omega_C^{n-k} = d (c_k (\omega_C) \wedge \omega_C^{n-k-1} \wedge \eta ) \). By Stokes, the local integrability follows from
	\begin{equation*}
		\int_{r = \epsilon} c_k (\omega_C) \wedge \omega_C^{n-k-1} \wedge \eta = O (\epsilon^{2(n-k)}) .
	\end{equation*}
\end{proof}
For a klt pair \((X, D)\)  with \(D=\sum_{j=1}^{N} (1-\beta_j) D_j\),  we define (at least when \(X\) and \(D_j\) are smooth), \( c_1(X, D) = c_1(X) - \sum_{j=1}^{N} (1-\beta_j) [D_j]  \) and \(	c_2 (X, D) \in H^4(X, \mathbb{R}) \) as:
\begin{equation} \label{logc2}
\begin{split}
	c_2 (X, D)  & = c_2 (X) + \sum_j (\beta_j-1) (-K_X \cdot D_j - D_j^2) \, + \\
	 & \sum_S (\nu_S - (1 + \sum_{S \subset D_j} (\beta_j -1) )) [S], 
\end{split}
\end{equation}
where \(S\) runs over the irreducible complex codimension two varieties obtained as intersections of the \(D_j\), \(\nu_S\) denotes the volume density at a generic point and \([S]\) is the Poincar\'e dual. In the simple normal crossing case our definition agrees with the one of Tian in \cite{tian} (Equation 2.6). The formula can be interpreted as a logarithmic Chern class with points weighted according to their volume densities. The first term corresponds to assigning volume density \(1\) at all points. In the second we subtract the singular strata in complex codimension one and re-add it  weighted by the volume density \(\beta_j\) at a generic point in \(D_j\). The term $1 + \sum_{S \subset D_j} (\beta_j -1)$ corresponds to \(S\) being counted with weight \(1\) in the first term and weight \(\beta_j -1\) for each divisor \(D_j\) that contains it. 

We expect that for KEcs metrics on \((X, D)\)  the following Bogomolov-Myaoka-Yau (BMY) inequality holds:
\begin{equation}
Prop(X, D) = \left( 2 (n+1) c_1(X, D)^2 - n c_2(X, D)  \right) [\omega]^{n-2} = \int_X |\mathring{Riem}|^2 \geq 0,
\end{equation}
with equality case characterizing constant holomorphic curvature. As a check, for the hyperplane arrangements considered in Propositions \ref{posandpk} and \ref{P5},  we have the following:

\begin{proposition}
	Let \( X = \mathbb{CP}^n \) and \( D = \sum_{i < j} \mu_{ij} H_{ij} \), then \[ n c_1 (X, D)^2 = 2 (n+1) c_2 (X, D) . \]
\end{proposition}

\begin{proof}
	We identify the even dimensional real cohomology groups of projective space with \(\mathbb{R}\) by sending the powers of the hyperplane section to \(1\). Then \( c_1 (X, D) = (n+1) - \sum_{i < j} \mu_{ij} = (n+1) (1 - S) \) with \(S = \sum_{j=1}^{n+2} \mu_j \). To compute \(c_2(X, D)\) note first that \( c_2 (\mathbb{CP}^n) = (n+1)n/2 \) and \(\sum_{i < j} \mu_{ij} (K_{\mathbb{CP}^n} \cdot H_{ij} + H_{ij}^2 ) = - n(n+1) S \). The singular codimension two strata is divided into two types: either \( S= H_{ij} \cap H_{kl} \) with \( \{i, j\} \cap \{k, l\} = \emptyset \) or \(S = H_{ij} \cap H_{ik} \cap H_{jk} \) with \(i<j<k\). We write \(\nu_{ijkl}\) and \(\nu_{ijk}\) for the corresponding volume densities at generic points, so that \( \nu_{ijkl} = (1-\mu_{ij}) (1-\mu_{kl}) \) and \( \nu_{ijk} = \gamma_{ijk}^2 \) with \( 2 \gamma_{ijk} = 2 - (\mu_{ij} + \mu_{ik} + \mu_{jk} ) \) equivalently \( \gamma_{ijk} = 1 - (\mu_i + \mu_j + \mu_k) \). The contribution of the codimension two strata in Equation \ref{logc2} is the sum of two terms, \(A + B\), with \( A = \sum_I (\nu_{ijkl} + \mu_{ij} + \mu_{kl} -1 ) = \sum_I \mu_{ij}\mu_{kl} \) where \(I\) is the index set of \((i, j, k, l)\) with \(i < j, k<l\), \( \{i, j\} \cap \{k, l\} = \emptyset \) and \(i <k\); and \(B = \sum_{i<j<k} (\gamma_{ijk}^2 +\mu_{ij}+ \mu_{ik} +\mu_{jk} -1 ) = \sum_{i<j<k} (\mu_i + \mu_j + \mu_k)^2 \). It is easy to see that
	\begin{equation*}
		A = \sum_I \mu_{ij} \mu_{kl} = \sum_I \mu_i \mu_k + \mu_i \mu_l + \mu_j \mu_k + \mu_j \mu_l = 2 \binom{n}{2} \sum_{i<k} \mu_i \mu_k
	\end{equation*}  
	
	\begin{equation*}
	B = \sum_{i<j<k} (\mu_i^2 + \mu_j^2 + \mu_k^2) + 2 \sum_{i<j<k} (\mu_i \mu_j + \mu_i \mu_k + \mu_j \mu_k ) = \binom{n+1}{2} \sum_{j=1}^{n+2} \mu_j^2 + 2n \sum_{i<j} \mu_i \mu_j .
	\end{equation*}
	Since \( \binom{n}{2} + n = \binom{n+1}{2} \) we obtain \( A + B = \binom{n+1}{2} S^2 \). We conclude that
	\begin{equation*}
		c_2(X, D) = \frac{(n+1)n}{2} (1-2S + S^2)
	\end{equation*}
	and therefore \( 2 (n+1) c_2 (X, D) = (n+1)^2 n (1-S)^2 = n c_1(X, D)^2 \).
\end{proof}

Given \(X\) and \( D_1, \ldots, D_N \subset X \) we obtain a semi-positive definite quadratic form in \( \mu = (\mu_1, \ldots, \mu_N) \) defined as \( Prop (X, \sum_{j=1}^{N}\mu_j D_j) \). If \(X = \mathbb{CP}^n\) then the constant term of the quadratic form vanishes, the hyperplane arrangement \(\mathcal{H}\) studied before has a subspace of dimension \(n+2\) in the null-space; when \(n=2\) the complete quadrilateral \(= \mathcal{H}\) is characterized as the one with maximal kernel (which is indeed 4 dimensional) among all line arrangements \cite{tretkoff}. There are other hyperplane arrangements (given by the mirrors of unitary reflection groups) that define quadratic forms with non-trivial kernel and also lead to constant holomorphic curvature metrics, different from the one discussed, see \cite{hirzebruch} and \cite{couwenbergheckmanlooijenga}. The BMY inequality leads to interesting relations between the volume densities of different strata. For example if \( c_1(X, D) = c_2 (X, D) = 0 \) then the weighted Euler characteristic should vanish. On the other hand we know a priori that certain singularities of the pair are forbidden if \((X, D)\) supports constant holomorphic curvature, for example an ordinary double point with no conical divisor going through it; one might then ask for an improved BMY inequality for pair having such singularities.

\bibliographystyle{plain}
\bibliography{localmodel}

\end{document}